\documentclass[12pt]{article}

\usepackage{amscd}
\usepackage{latexsym,amsmath,amssymb,amsbsy,amsthm}
\usepackage[square]{natbib}

\usepackage{hyperref}
\hypersetup{
    unicode=false,          
    pdftoolbar=true,        
    pdfmenubar=true,        
    pdffitwindow=false,     
    pdfstartview={FitH},    
    pdftitle={},    
    pdfauthor={Soumendu Sundar Mukherjee},     
    pdfsubject={Mathematics},   
    pdfcreator={Creator},   
    pdfproducer={Producer}, 
    pdfkeywords={keyword1} {key2} {key3}, 
    pdfnewwindow=true,      
    colorlinks=true,       
    linktoc=page,
    linkcolor=blue,          
    citecolor=blue,        
    filecolor=green,      
    urlcolor=cyan,           
    linkbordercolor={1 1 1},
    citebordercolor={0 1 0},
    urlbordercolor={0 1 1}
}

\usepackage{lscape,fancyhdr,fancybox}
\usepackage{graphicx,epsfig, xy}
\usepackage{color}
\usepackage{amsfonts}
\usepackage[hmarginratio=1:1, vmarginratio =1:1,
textheight=22cm, textwidth=15cm]{geometry}

\usepackage[title,titletoc,toc]{appendix}

\numberwithin{equation}{section}

\DeclareMathOperator{\tr}{tr}

\theoremstyle{plain}

\theoremstyle{plain}
\newtheorem{theorem}{Theorem}[section]
\theoremstyle{plain}
\newtheorem{lemma}{Lemma}[section]
\theoremstyle{plain}

\theoremstyle{plain}

\theoremstyle{remark}
\newtheorem{remark}{Remark}[section]
\theoremstyle{remark}

\newcommand{\E}{\mathbb{E}}
\newcommand{\R}{\mathbb{R}}
\newcommand{\C}{\mathbb{C}}

\begin{document}

\def\shorttitle{Bulk behaviour of skew-symmetric patterned random matrices}
 \def\shortauthors{Arup Bose and Soumendu Sundar Mukherjee}
 \fancyhf{}
 \renewcommand\headrulewidth{0pt}
 \fancyhead[C]{
 \ifodd\value{page}
   \small\scshape\shortauthors
 \else
   \small\scshape\shorttitle
 \fi
 }
 \fancyhead[L]{
   \ifodd\value{page}
     \thepage
   \fi
   }
 \fancyhead[R]{
   \ifodd\value{page}
   \else
      \thepage
     \fi
     }
 \pagestyle{fancy}
  
\title{\textbf{\LARGE \sc
Bulk behaviour of skew-symmetric patterned random matrices}}
 \author{
 \parbox[t]{0.50\textwidth}{{\sc Arup Bose}
 \thanks{Research  supported by J.C.~Bose National Fellowship, Dept.~of Science and Technology, Govt.~of India.}
 \\ {\small Statistics and Mathematics Unit\\ Indian Statistical Institute\\ 203 B. T.~Road, Kolkata 700108\\ INDIA\\  E-mail: \texttt{bosearu@gmail.com}\\}}
\parbox[t]{0.50\textwidth}{{\sc Soumendu Sundar Mukherjee}
  \\ {\small Master of Statistics student\\ Indian Statistical Institute\\ 203 B. T.~Road, Kolkata 700108\\ INDIA\\  E-mail: \texttt{soumendu041@gmail.com}\\}}}

\date{This version February 14, 2014}

\maketitle
\begin{abstract}
Limiting Spectral Distributions (LSD) of real symmetric patterned matrices have been well-studied. In this article, we consider skew-symmetric/anti-symmetric patterned random matrices and establish the LSDs of several common matrices. For the skew-symmetric Wigner, skew-symmetric Toeplitz and the skew-symmetric Circulant, the LSDs (on the imaginary axis) are the same as those in the symmetric cases. For the skew-symmetric Hankel and the skew-symmetric Reverse Circulant however, we obtain new LSDs. We also show the existence of the LSDs for the triangular versions of these matrices. 

 We then introduce a related modification of the symmetric matrices by changing the sign of the lower triangle part of the matrices. In this case, the modified Wigner, modified Hankel and the modified Reverse Circulants have the same LSDs as their usual symmetric counterparts while new LSDs are obtained for the modified Toeplitz and the modified Symmetric Circulant.  
\end{abstract}
   \medskip

\medskip

\bigskip
\bigskip

\section{Introduction} 
Suppose $A_n$ is an  $n\times n$ 
matrix with eigenvalues $\lambda_1, \ldots , \lambda_n$.
The empirical spectral measure $\mu_n$ of $A_n$ is the
random
measure 
\begin{equation}
\mu_n= \frac{1}{n} \sum_{i=1}^n \delta_{ \lambda_i},
\end{equation}
where $\delta_{x}$ is the Dirac delta measure at $x$. The
corresponding random
probability distribution function is known as the \emph{Empirical
Spectral Distribution} (ESD) and is denoted by
$F^{A_n}$.

The sequence $\{F^{A_n}\}$ is said to converge (weakly) almost surely to a non-random 
distribution function $F$ if, outside a null set, as $n \to \infty$, $F^{A_n}(\cdot) \rightarrow F(\cdot)$ at all
continuity points of $F$. $F$ is known as the \emph{Limiting Spectral Distribution} (LSD). 

There has been a lot of recent work on obtaining the LSDs of
large dimensional patterned random matrices. These matrices may be defined as follows. Let $\{ x_i; i \geq 0\}$ be a sequence of random variables, called an  \textit{input sequence}. Let 
 $\mathbb Z$ be the set of all integers and let 
 $\mathbb Z_{+}$ be the set of all non-negative integers. Let
\begin{equation}\label{eq:link}
L_n: \{1, 2, \ldots n\}^2 \to \mathbb Z \ \text{(or} \ \ \mathbb Z^2) \  n \geq 1, 
\end{equation} be a sequence of functions. We shall
write $L_n=L$ and call it the \textit{link} function and by abuse of
notation we write $\mathbb Z_{+}^2$ as the common domain of $\{ L_n
\}$. Matrices  of the form
\begin{equation}\label{eq:patternmatrix}
A_n=n^{-1/2}((x_{L(i,j)}))_{1\leq i, j \leq n}
\end{equation}
are called \emph{patterned matrices}. If
$L(i,j)=L(j,i)$ for all $i,j$, then the matrix is symmetric. In this article, we shall denote the LSD of $\{n^{-1/2}A_n\}$, if it exists, by $\mathcal{L}_{A}$. 

The symmetric patterned matrices that have received particular attention in the literature are the Wigner, Toeplitz, Hankel,  Reverse Circulant and the Symmetric Circulant matrices. Their link functions are given in Table \ref{table1}. 
\begin{table}[htbp]
\centering
\begin{tabular}{|c|c|c|}
\hline
Matrix & Notation & Link function\\
\hline
Wigner & $W_n$ & $L_W(i,j)=(\min\{i,j\},\max\{i,j\})$\\
Toeplitz & $T_n$ &$L_T(i,j)=|i-j|$\\
Hankel & $H_n$ & $L_H(i,j)=i+j$\\
Symmetric Circulant & $SC_n$ & $L_{SC}(i,j)=\frac{n}{2}-|\frac{n}{2}-|i-j||$\\
Reverse Circulant & $RC_n$ & $L_{RC}(i,j)=(i+j)(\text{mod } n)$\\
\hline
\end{tabular}
\vskip10pt
\caption{Some common symmetric patterned matrices and their link functions.}
\label{table1}
\end{table}

While the LSDs of the Wigner, Reverse Circulant and the Symmetric Circulant are known explicitly, very little is known about the LSDs of the Hankel and the Toeplitz. LSD existence is also known for the upper triangular versions of these matrices, though the nature of these limits is not known.

The  LSD of the non-symmetric Wigner (the i.i.d.~matrix) is the circular law (uniform measure on the unit disc in $\C$) and for the Circulant matrix the LSD is  bivariate Gaussian. It is not known whether LSDs exist for non-symmetric Toeplitz and Hankel matrices, even though simulation evidence is positive. See \citet{bryc2006spectral,bose2008another}.
It appears to be difficult to establish the LSD for these non-symmetric matrices. 

In this article, we study the existence of the LSDs of skew-symmetric/anti-symmetric patterned matrices. (In the Physics literature the term ``anti-symmetric'' is more common. Technically, if S is a skew-symmetric matrix, then $iS$ is called an anti-symmetric matrix, where $i$ is the imaginary unit. Note that $iS$ is Hermitian.) Anti-symmetric Gaussian matrices appeared in the classic work of \citet{mehta2004random} who, among other things, gave an expression for the joint distribution of the eigenvalues. Singular values of skew-symmetric Gaussian Wigner matrices are useful in Statistics too, e.g., in paired comparisons model (see \citet{kuriki1993orthogonally,kuriki2010distributions}). Recently, \citet{dumitriu2010tridiagonal} obtained tridiagonal realizations of anti-symmetric Gaussian $\beta$-ensembles.  

We first establish the existence of the LSDs of several real skew-symmetric patterned random matrices and identify the limits in some cases. For the skew-symmetric Wigner, skew-symmetric Toeplitz and the skew-symmetric Circulant, the LSDs (on the imaginary axis) are the same as those in the symmetric cases. However, for the skew-symmetric Hankel and the skew-symmetric Reverse Circulant, we obtain new LSD. We also show the existence of the LSDs for the triangular versions of these matrices (introduced in \citet{basu2012spectral}).

We also introduce a related modification of the symmetric matrices by changing the sign of the lower triangle part of the matrices. In this case, the modified Wigner, modified Hankel and the modified Reverse Circulant have the same LSD as their usual symmetric counterparts whereas new LSD are obtained for the modified Toeplitz and the modified Symmetric Circulant.

\section{Preliminaries}
We shall use the method of moments to establish the existence of the LSD. For any matrix $A$, let $\beta_h(A)$ denote the $h$-th moment of the ESD of $A$. 
We quote the following lemma which is easy to prove.
\begin{lemma}\label{mainlemma}
Let $\{A_n\}$ be a sequence of random matrices with all real eigenvalues. Suppose there exists a sequence $\{\beta_h\}$ such that\vskip5pt 

\noindent 
(i)  for every $h\geq 1$, $\E(\beta_h(A_n))\rightarrow \beta_h$, \vskip5pt

\noindent (ii)  $\sum_{n=1}^{\infty}\E[\beta_h(A_n)-\E(\beta_h(A_n))]^4<\infty$ for every $h\geq 1$ and \vskip5pt 

\noindent 
(iii) the sequence $\{\beta_h\}$ satisfies Carleman's condition,  $\sum\beta_{2h}^{-1/2h}=\infty$. \vskip5pt

Then  the LSD of $F^{A_n}$ exists and equals $F$ with moments $\{\beta_h\}$.
\end{lemma}
To prove the existence of any LSD, we shall make use of the general notation and theory developed in \citet{bose2008another} for patterned matrices. 
First observe that all the link functions satisfy the so called 
Property B: the total number of times any particular variable appears in any row is uniformly bounded. Moreover, the total number of different variables in the matrix and the total number of times any variable appears in the matrix are both of the order $n$. This implies that the general theory applies to this class of link functions.

We shall consider the following assumptions on the input random variables.
\vskip10pt

\noindent {\bf (A1).} The input random variables are independent and uniformly bounded with mean $0$, and variance $1$.\vskip10pt
\noindent {\bf (A2).} The input random variables are i.i.d. with mean $0$ and variance $1$.\vskip10pt
\noindent {\bf (A3).} The input random variables are independent with mean $0$ and variance $1$, and with uniformly bounded moments of all orders.\vskip10pt

In particular,  if the LSD exists under Assumption (A1), then the same LSD continues to hold under Assumptions (A2) or (A3). Thus in our arguments, without loss of any generality, Assumption (A1) is assumed to hold.  
Traditionally, LSD results are stated under Assumption (A1) and Assumption (A3) is appropriate while studying the joint convergence of more than one sequence of matrices.

The \emph{Moment-Trace Formula} plays a key role in this approach. A function $$\pi : \{0, 1, \cdots, h\} \rightarrow \{1,2,\cdots,n\}$$ with $\pi(0) = \pi(h)$ is  called a \emph{circuit} 
of
length $h$. The dependence of a circuit on $h$ and $n$ is suppressed. Then
\begin{equation}
\beta_h(A)= \frac{1}{n}\tr(A^h)=\frac{1}{n}\sum_{\pi \text{ circuit of length $h$}}a_\pi,
\end{equation}
where 
\[
a_{\pi} := a_{L(\pi(0),\pi(1))}a_{L(\pi(1),\pi(2))}\ldots a_{L(\pi(h-1),\pi(h))}.
\]

If $L(\pi(i-1),\pi(i))=L(\pi(j-1),\pi(j))$, with $i<j$, we shall use the notation $(i,j)$ to denote such a match of the $L$-values.
From the general theory, it follows that  circuits where there are only pair-matches are relevant when computing limits of moments. 

Two circuits $\pi_1$ and $\pi_2$ are equivalent if and only if their $L$-values respectively match at the
same locations, i.e., if for all $i, j$,
\[
L(\pi_1(i-1),\pi_1(i))=L(\pi_1(j-1),\pi_1(j)) \Leftrightarrow L(\pi_2(i-1),\pi_2(i))=L(\pi_2(j-1),\pi_2(j)).
\]

Any equivalence class can be indexed by a partition of $\{1,2,\cdots, h\}$. We label these partitions by \emph{words} of length $h$ of letters where the first occurrence
of each letter is in alphabetical order. For example, if $h = 4$ then the partition
$\{\{1,3\},\{2,4\}\}$ is represented by the word $abab$. This identifies all circuits $\pi$ for
which $L(\pi(0),\pi(1)) = L(\pi(2), \pi(3))$ and $L(\pi(1),\pi(2)) = L(\pi(3),\pi(1))$.
Let $w[i]$ denote the $i$-th entry of $w$. The equivalence class corresponding to $w$ is 
\[
\Pi(w) := \{\pi \mid w[i]=w[j]\Leftrightarrow L(\pi(i-1),\pi(i))=L(\pi(j-1),\pi(j))\}.
\]
By varying $w$, we obtain all the equivalence classes. It is important to note that for any
fixed $h$, even as $n\rightarrow \infty$, the number of words (equivalence classes) remains finite but
the number of circuits in any given $\Pi(w)$ may grow indefinitely. Henceforth we shall denote the set of all words of length $h$ by $\mathcal{A}_{h}$. 

Notions of matches carry over to words. A word is \emph{pair-matched} if every letter 
appears exactly twice in that word. The set of all pair-matched words of length $2k$
is  denoted by $\mathcal{W}_{2k}$.
For technical reasons it is often easier to deal with a class larger than $\Pi(w)$:
\[
\Pi^*(w) = \{\pi \mid w[i]=w[j]\Rightarrow L(\pi(i-1),\pi(i))=L(\pi(j-1),\pi(j))\}.
\]

Any $i$ (or $\pi(i)$ by abuse of notation) is a \emph{vertex}. It is \emph{generating} if either $i = 0$
or $w[i]$ is the first occurrence of a letter. Otherwise, it is called non-generating. For
example, if $w = abbcab$ then $\pi(0), \pi(1), \pi(2), \pi(4)$ are generating and $\pi(3), \pi(5), \pi(6)$
are non-generating. The set of generating vertices (indices) is denoted by $S$. By Property B, a circuit is completely determined, up to finitely many choices, by its generating vertices.

Note that from the general theory it follows that the LSD exists if for each $w \in \mathcal{W}_{2k}$, the following limit exists:
$$p(w)=\lim n^{-(k+1)}\# \Pi^*(w).$$

\section{A unified framework for real skew-symmetric matrices}
If $A$ is an $n\times n$ skew-symmetric matrix, then all its eigenvalues $\{\lambda_j\}$ are purely imaginary (and has one zero eigenvalue when $n$ is odd), and every eigenvalue occurs in conjugate pairs.
Therefore one can \emph {define} an empirical spectral distribution of $A$ on $\R$ as 
\begin{equation}
F^A(x):=\frac{1}{n}\sum_{j=1}^{n}\mathbf{1}_{(i\lambda_j \leqslant x)}
\end{equation}
and moreover,  $F^A$ is a symmetric distribution. 
Therefore, in order to apply the moment method, it suffices to deal with only the even moments. Note that
\begin{align*}
\beta_{2k}(A) & = \int x^{2k} \, dF^A(x) \\
			& = \frac{1}{n}\sum_{j=1}^{n}(i\lambda_j)^{2k} \\
			& = (-1)^k \frac{1}{n}\sum_{j=1}^{n} \lambda_j^{2k} \\
			& = (-1)^k \frac{1}{n}\tr(A^{2k}). 
\end{align*}

Let $\{A_n\}$ be a sequence of $n\times n$ patterned random matrices with the symmetric link function $L$.  
Let   
\begin{equation}
s_{ij}=(1-\delta_{ij})(-1)^{\mathbf{1}_{(i>j)}},
\end{equation}
where $\delta_{ij}$ is the Kronecker-delta.
Let $S_n=((s_{ij}))$ be the $n \times n$ matrix 
\begin{equation}
S_n=
\begin{pmatrix}
0 & 1 & \hdots & 1 \\
-1 & 0 & \hdots & 1 \\
\vdots & \vdots & \ddots & \vdots \\
-1 & -1 & \hdots & 0
\end{pmatrix}_{n\times n}.
\end{equation}
Then we can construct $\widetilde{A}_n$, the skew-symmetric version of $A_n$ by
\begin{equation}
\widetilde{A}_n=S_n\odot A_n,
\end{equation}
where $\odot$ denotes the Schur-Hadamard/entrywise product.

We shall without loss of generality assume that (A1) holds. The moment-trace formula for $\widetilde{A}_n$ may be written as
\begin{equation}
\beta_{2k}(n^{-1/2}\widetilde{A}_n)=(-1)^k\frac{1}{n^{1+k}}\sum_{\pi \ \text{circuit of length $2k$}} s_{\pi} a_{\pi}.
\end{equation}
Therefore
\begin{equation}\label{eq:emt1}
\E\beta_{2k}(n^{-1/2}\widetilde{A}_n)=(-1)^k\frac{1}{n^{1+k}}\sum_{\pi \ \text{circuit of length $2k$}} s_{\pi}\E a_{\pi}.
\end{equation}
Using the concept of words we may rewrite (\ref{eq:emt1}) as
\begin{equation}\label{eq:emt2}
\E\beta_{2k}(n^{-1/2}\widetilde{A}_n)=(-1)^k\frac{1}{n^{1+k}}\sum_{w \in \mathcal{A}_{2k}}\sum_{\pi \in \Pi(w)} s_{\pi}\E a_{\pi}.
\end{equation}

Suppose $L$ satisfies Property B. Let $C^{L}_{h,3+}$ denote the set of $L$-matched $h$-circuits on $\{1,\cdots , n\}$ with at least one edge of order $\geqslant 3$. Then Lemma 1(a) of \citet{bose2008another} says that there is a constant $C$ depending on $L$ and $h$ such that
\[
\# C^{L}_{h,3+}  \leqslant C n^{\lfloor(h+1)/2\rfloor}.
\]
Combining this with the observation that $|s_{\pi}|\leqslant 1$ it is easy to see that
\begin{equation}
\lim_n\frac{1}{n^{1+k}}\sum_{\pi \in C^{L}_{2k,3+}} s_{\pi}\E a_{\pi}=0.
\end{equation}
Therefore
\begin{equation}\label{eq:emt3}
\lim_n \E\beta_{2k}(n^{-1/2}\widetilde{A}_n)=(-1)^k \lim \frac{1}{n^{1+k}}\sum_{w \in \mathcal{W}_{2k}}\sum_{\pi \in \Pi(w)} s_{\pi}\E a_{\pi}.
\end{equation}

Noting that under our assumption $\E a_{\pi}=1$ for any pair-matched circuit $\pi$, (\ref{eq:emt3}) reduces to
\begin{equation}\label{eq:emt5}
\lim \E\beta_{2k}(n^{-1/2}\widetilde{A}_n)=(-1)^k\sum_{w \in \mathcal{W}_{2k}}\lim_n\frac{1}{n^{1+k}}\sum_{\pi \in \Pi(w)} s_{\pi},
\end{equation}
\textit{provided the limits in the right side  exist}. In fact, since $\Pi^*(w)\setminus \Pi(w) \subseteq C^L_{2k,3+}$, one has
\[
\lim_n\frac{1}{n^{1+k}}\sum_{\pi \in \Pi(w)} s_{\pi}=\lim_n\frac{1}{n^{1+k}}\sum_{\pi \in \Pi^*(w)} s_{\pi},
\]
and thus one can write
\begin{equation}\label{eq:emt6}
\lim \E\beta_{2k}(n^{-1/2}\widetilde{A}_n)=(-1)^k\sum_{w \in \mathcal{W}_{2k}}\lim_n\frac{1}{n^{1+k}}\sum_{\pi \in \Pi^*(w)} s_{\pi},
\end{equation}
provided the limits exist for each $w$. If we define 
\[
p_{\widetilde{A}}(w):=(-1)^k\lim_n\frac{1}{n^{1+k}}\sum_{\pi \in \Pi(w)} s_{\pi},
\]
then (\ref{eq:emt6}) becomes
\begin{equation}\label{eq:emt7}
\lim \E\beta_{2k}(n^{-1/2}\widetilde{A}_n)=\sum_{w \in \mathcal{W}_{2k}}p_{\widetilde{A}}(w).
\end{equation}
In this context, we recall the analogous expression for symmetric matrices $A_n$ from \citet{bose2008another}:
\[
\lim \E\beta_{2k}(n^{-1/2}A_n)=\sum_{w \in \mathcal{W}_{2k}}p_{A}(w),
\]
where
\[
p_{A}(w):=\lim_n\frac{1}{n^{1+k}}\#\Pi(w)=\lim_n\frac{1}{n^{1+k}}\#\Pi^*(w)
\]
is assumed to exist for each $w\in\mathcal{W}_{2k}$.

It is not difficult to show that if the limits exist in (\ref{eq:emt6}), then  Condition (iii) of Lemma~\ref{mainlemma} follows (see Theorem~3 of \citet{bose2008another} for the argument in the symmetric case; in the skew-symmetric case too, one can use their argument verbatim because $|s_{\pi}|\leqslant 1$). In fact, the limiting moments are sub-Gaussian. The verification of Condition (ii) is also easy since 
\begin{equation*}
\prod_{j=1}^{4} \E(s_{\pi_j}a_{\pi_j}-\E s_{\pi_j}a_{\pi_j})
 =s_{\pi_1}s_{\pi_2}s_{\pi_3}s_{\pi_4}\prod_{j=1}^{4}\E(a_{\pi_j}-\E a_{\pi_j})
\end{equation*} 
and the arguments given in the proof of Lemma 2 of \citet{bose2008another} apply with minor modifications.

In the next section, we shall consider several skew-symmetric patterned matrices and show that Condition (i) of Lemma~\ref{mainlemma} holds by  arguing that (\ref{eq:emt5}) holds in each case. 

\section{Some specific matrices}
 First note that
\begin{equation}\label{eq:s}
 s_{\pi}=(-1)^{\sum_{j=1}^{2k} \mathbf{1}_{(\pi(j-1)>\pi(j))}}\prod_{j=1}^{2k}(1-\delta_{\pi(j-1),\pi(j)}).
\end{equation}
It is convenient to use some graph theoretic terminology to deal with (\ref{eq:s}).  Consider the complete directed graph $DK_n$ on $V=\{1,\cdots\,n\}$. Note that $\pi$ defines a directed circuit of length $2k$ on this graph. Call the numerical value  of each vertex its \emph{level}. Associate with each $\pi$ a marking-vector $(\epsilon_1,\cdots,\epsilon_{2k})$, where
\begin{equation}
\epsilon_j=(-1)^{\mathbf{1}_{(\pi(j-1)>\pi(j))}}(1-\delta_{\pi(j-1),\pi(j)}).
\end{equation}
Note that if a traveler moves along the circuit $\pi$, starting from $\pi(0)$, and marks each move $\pi(j-1) \rightsquigarrow \pi(j)$ by $\epsilon_j$, then moving to a higher (respectively lower)  level corresponds to a mark of $1$ (respectively $-1$) and remaining at the same level corresponds to marking with $0$. Then
\begin{equation}
s_{\pi}=\prod_{j=1}^{2k}\epsilon_j.
\end{equation}
Note that a circuit $\pi$ contains a \emph{loop} if and only if $s_{\pi}=0$.
\subsection{LSD of $n^{-1/2}\widetilde{W}_n$}
We recall the concept of \emph{Catalan} words from \citet{bose2008another}. A Catalan word of length $2$ is just a double letter $aa$. In general, a Catalan word of length $2k$, $k>1$, is a word $w\in \mathcal{W}_{2k}$ containing a double letter such that if one deletes the double letter the reduced word becomes a Catalan word of length $2k-2$.
For example, $abba$, $aabbcc$, $abccbdda$ are Catalan words whereas $abab$, $abccab$, $abcddcab$
are not. The set of all Catalan words of length $2k$ will be denoted by $\mathcal{C}_{2k}$. It is known that \begin{equation}
\#\mathcal{C}_{2k}=\frac{1}{k+1}\binom{2k}{k},
\end{equation}
the ubiquitous Catalan number from Combinatorics. It is known that $\#\mathcal{C}_{2k}$ also equals the $2k$-th moment of the semicircle law, the LSD of the Wigner matrix.
\begin{theorem}\label{thm:wig}
Suppose that the entries of the  skew-symmetric Wigner $n^{-1/2}\widetilde{W}_n$ satisfy (A1) or (A2) or (A3). Then its  LSD is the semi-circular law almost surely. 
\end{theorem}
\noindent 
\begin{proof} It is well known (see, e.g., \citet{bose2008another}) that for the symmetric Wigner matrix only Catalan words contribute in the limit. In fact, one has
\begin{equation}\label{eq:wig}
p_W(w)=\lim_n\frac{1}{n^{1+k}}\#\Pi^*(w)=
\begin{cases}
0, & \text{ if $w\notin \mathcal{C}_{2k}$} \\
1, & \text{ if $w \in \mathcal{C}_{2k}$}.
\end{cases}
\end{equation}
From (\ref{eq:wig}) and the fact that $|s_{\pi}|\leqslant 1$ it follows that
\begin{equation}\label{eq:skew-wig}
|p_{\widetilde{W}}(w)|
\begin{cases}
=0, & \text{if $w\in \mathcal{W}_{2k}\setminus\mathcal{C}_{2k}$} \\
\leqslant 1, & \text{if $w \in \mathcal{C}_{2k}$}.
\end{cases}
\end{equation}
We shall prove that if $w$ is a Catalan word then $p_{\widetilde{W}}(w)$ exists and equals $1$. Then (\ref{eq:emt6}) would imply that
\begin{equation}
\E\beta_{2k}(n^{-1/2}B_n)=\#\mathcal{C}_{2k},
\end{equation}
establishing the semi-circle limit for the ESD of $\{n^{-1/2}\widetilde{W}_n\}$. 

We first observe that if we  replace the diagonal entries by $0$, the LSD does not change. It follows from this observation that \emph{circuits with loops together do not have any contribution to $p_{\widetilde{W}}(w)$}. It now suffices for our purpose to prove that if $w\in\mathcal{C}_{2k}$ and $\pi\in\Pi^*(w)$, then
\begin{equation}\label{eq:spi_wig}
s_{\pi}=\begin{cases}
(-1)^k, & \text{ if $\pi$ is loopless} \\
0, & \text{ otherwise}.
\end{cases}
\end{equation}
To prove this,  suppose that a double letter appears  at the $i$-th and the $(i+1)$-th positions. Consider a loopless $\pi\in\Pi^*(w)$. Since, $w[i]=w[i+1]$, we must have
\[
L_W(\pi(i-1),\pi(i))=L_W(\pi(i),\pi(i+1)).
\]
Since $\pi$ is loopless,  it follows that we must have $\pi(i-1)=\pi(i+1)\neq\pi(i)$. There are two possibilities: either $\pi(i-1)<\pi(i)$ or $\pi(i-1)>\pi(i)$. In the first case $\epsilon_i=1$ and $\epsilon_{i+1}=-1$ while in the second case $\epsilon_i=-1$ and $\epsilon_{i+1}=-1$. In either case we have
\[
\epsilon_i\epsilon_{i+1}=-1.
\]
Now delete the double letter and think of $\pi$ as a circuit of length $2k-2$ by identifying the vertices $(i-1)$ and $(i+1)$ as identical and deleting the vertex $i$. The resulting word $w'$ is still Catalan and the resulting circuit $\pi'$ is loopless and lies in $\Pi^*(w')$. Apply the above procedure again. Clearly, we will need $k$ iterations of this procedure to empty the word $w$ and each such iteration contributes one $-1$, which proves (\ref{eq:spi_wig}) and hence the theorem.
\end{proof}

\begin{remark}
\citet{basu2012spectral} considered  upper/lower triangular versions of the Wigner, $W_n^{\Delta}$. Its LSD $\mathcal{L}_{W^{\Delta}}$ is different from the semi-circular law, but its free convolution with itself is the semi-circular law. It follows from the proof of Theorem~\ref{thm:wig} and their moment calculations that the LSD of the skew-symmetric triangular Wigner $\widetilde{W}^{\Delta}$ is again $\mathcal{L}_{W^{\Delta}}$. 
\end{remark}
\subsection{LSD of $n^{-1/2}\widetilde{T}_n$}
The LSD of the symmetric Toeplitz matrix $T_n$ was first established by \citet{bryc2006spectral}. The properties of the limit law $\mathcal{L}_T$ are not well understood. We shall consider the skew-symmetric Toeplitz $\widetilde{T}_n$ and show analogous to the Wigner case that the LSD is $\mathcal{L}_T$.   
\begin{theorem}\label{thm:toep}
Suppose the entries of the skew-symmetric Toeplitz $n^{-1/2}\widetilde{T}_n$ satisfy (A1) or (A2) or (A3). Then its LSD  is $\mathcal{L}_T$, the LSD of the symmetric Toeplitz. 
\end{theorem}
\begin{proof}
Let $w\in\mathcal{W}_{2k}$ and $s(i):=\pi(i)-\pi(i-1)$. Define
\begin{equation}
\Pi^{**}(w):=\{\pi \, \mid \, w[i]=w[j]\Rightarrow  s(i)+s(j)=0\}.
\end{equation}
Then \citet{bose2008another} show that
\begin{equation}\label{eq:toep}
p_T(w)=\lim_n\frac{1}{n^{1+k}}\#\Pi^*(w)=\lim_n\frac{1}{n^{1+k}}\#\Pi^{**}(w).
\end{equation}
As in the Wigner case circuits with loops do not contribute and to establish our goal it suffices to prove that if $w\in\mathcal{W}_{2k}$ and $\pi\in\Pi^{**}(w)$, then
\begin{equation}\label{eq:spi_toep}
s_{\pi}=\begin{cases}
(-1)^k, & \text{ if $\pi$ is loopless} \\
0, & \text{ otherwise}.
\end{cases}
\end{equation}
The proof of this is much easier than the Wigner case as all the difficulty is relegated to the proof of (\ref{eq:toep})). Consider a loopless circuit $\pi \in \Pi^{**}(w)$. Note that $w[i]=w[j]$ implies that $s(i)+s(j)=0$ and since $\pi$ is loopless, we have
\[
s(i)s(j)=-s(j)^2<0.
\]
This immediately implies that
\[
\epsilon_i\epsilon_{j}=(-1)^{\mathbf{1}_{(s(i)<0)}+\mathbf{1}_{(s(j)<0)}}=-1.
\]
Since $w$ is pair-matched, there are exactly $k$ matches from each of which comes one $-1$. This establishes (\ref{eq:spi_toep}) completes the proof. 
\end{proof}

\begin{remark}
\citet{basu2012spectral} considered upper/lower triangular versions of the Toeplitz, $T_n^{\Delta}$. They proved the existence of the LSD but it could not be identified. It follows from the proof of Theorem~\ref{thm:toep} and their moment calculations that the LSD of the skew-symmetric triangular Toeplitz $\widetilde{T}^{\Delta}$ is again $\mathcal{L}_{T^{\Delta}}$, exactly paralleling the Wigner case.
\end{remark}
\begin{remark}
Recently, \citet{sen2013top} have shown that the top eigenvalue of the symmetric random Toeplitz matrix scaled by $\sqrt{n\log n}$ converges in $L_{2+\epsilon}$ to a constant when the entries have uniformly bounded $(2+\epsilon)$-th moment, $\epsilon>0$. Modifying their arguments suitably one can prove the same result for the top eigenvalue of the skew-symmetric random Toeplitz matrix.
\end{remark}
\subsection{LSDs of $n^{-1/2}\widetilde{SC}_n$ and $n^{-1/2}\widetilde{PT}_n$}
\citet{massey2007dist} defined a (symmetric) matrix to be palindromic if its first row is a palindrome. See \citet{bose2008another} for a moment method proof of the fact that the Symmetric Circulant matrix $SC_n$ and the Palindromic Toeplitz matrix  $PT_n$ have the standard Gaussian distribution on $\R$ as their LSDs. We show that the corresponding skew-symmetric versions $\widetilde{SC}_n$ and $\widetilde{PT}_n$ also have the same LSDs. 
\begin{theorem}\label{thm:symcirc+paltoep}
If the entries of  $n^{-1/2}\widetilde{SC}_n$ and $n^{-1/2}\widetilde{PT}_n$ satisfy (A1) or (A2) or (A3), then their  LSD  is $\mathcal{N}(0,1)$, the standard Gaussian distribution on $\R$. 
\end{theorem}
\begin{proof} We first tackle $\widetilde{SC}_n$.  
From \citet{bose2008another}, it is known that for any $w\in \mathcal{W}_{2k}$ if one defines
\begin{equation}
\Pi'(w):=\{\pi \, \mid \, w[i]=w[j]\Rightarrow  s(i)+s(j)=0,\pm n\},
\end{equation}
then one actually has
\begin{equation}
p_{SC}(w)=\lim_n\frac{1}{n^{1+k}}\#\Pi^*(w)=\lim_n\frac{1}{n^{1+k}}\#\Pi^{'}(w)=1.
\end{equation}
 Once again circuits with loops have no role to play and to prove the desired result it suffices to prove that if $w\in\mathcal{W}_{2k}$ and $\pi\in\Pi^{'}(w)$, then
\begin{equation}\label{eq:spi_symcirc}
s_{\pi}=\begin{cases}
(-1)^k, & \text{ if $\pi$ is loopless} \\
0, & \text{ otherwise}.
\end{cases}
\end{equation}
Due to the similarity with the Toeplitz link function,  the proof of the above  is similar to that in the Toeplitz case. Let $\pi$ be a loopless circuit from $\Pi'(w)$. Suppose that $w[i]=w[j]$. Then we have $s(i)+s(j)=0,\pm n$. We treat each of these three cases separately:

\begin{enumerate}
\item $s(i)+s(j)=0$. This is same as the Toeplitz case and we conclude that 
$\epsilon_i\epsilon_j=-1.$
\item $s(i)+s(j)=n$. Note that $s(i)=n-s(j)$ and since $\pi$ is loopless,
\[
|s(j)|=|\pi(j)-\pi(j-1)|\leqslant n-1.
\]
Therefore, $s(i)=n-s(j)>0$. By symmetry, $s(j)>0$. Therefore, in this case
$\epsilon_i\epsilon_j=1.$
\item $s(i)+s(j)=-n$. Note that $s(i)=-(n+s(j))$, and therefore  $s(i)$, and by symmetry $s(j)$,  are  both negative ceding $\epsilon_i\epsilon_j=1.$
\end{enumerate}
Therefore, combining the above, 
\[
s_{\pi}=(-1)^{k-e_{\pi}},
\]
where $e_{\pi}$ is the number of matches $(i,j)$ where $s(i)+s(j)=\pm n$. It suffices to show that $e_{\pi}$ is even. But note that
\[
\sum_{i=1}^{2k}s(i)=\pi(2k)-\pi(0)=0,
\]
which cannot occur unless $e_{\pi}$ is even. 
This establishes (\ref{eq:spi_symcirc}) and completes the proof for $\widetilde{SC}_n$.\vskip5pt

To prove the same for $\widetilde{PT}_n$ we take the approach of \citet{bose2008another}. We need the following version of the well known interlacing inequality. We omit its proof.\vskip5pt
 
Suppose $A$ is a real skew-symmetric matrix with eigenvalues $i\lambda_j$ with $\lambda_1\geqslant \lambda_2 \geqslant \cdots \geqslant \lambda_n$. Let $B$ be the $(n-1)\times (n-1)$ principal submatrix of $A$ with eigenvalues $i\mu_k$ with $\mu_1 \geqslant \mu_2 \cdots \geqslant \mu_{n-1}$. Then one has
\[
\lambda_1 \geqslant \mu_1 \geqslant \lambda_2 \geqslant \mu_2 \geqslant \cdots \geqslant \mu_{n-1} \geqslant \lambda_n,
\]
in other words, the imaginary parts of the eigenvalues of $B$ are interlaced between the imaginary parts of the eigenvalues of $A$.  \vskip5pt

As a consequence 
\begin{equation}\label{eq:linfty}
||F^A-F^B||_{\infty}\leqslant \frac{1}{n}.
\end{equation}
Now note that the $n\times n$ principal submatrix of $\widetilde{SC_{n+1}}$ is $\widetilde{PT}_n$. Therefore, from (\ref{eq:linfty}) we can conclude that $\widetilde{PT}_n$ also has the standard Gaussian law as its LSD.
\end{proof}
\begin{remark}
\citet{basu2012spectral} considered the  upper/lower triangular versions of the symmetric Circulant, $SC_n^{\Delta}$. They proved the existence of the LSD but it could not be identified. It follows from the proof of Theorem~\ref{thm:symcirc+paltoep} and their moment calculations that the LSD of the skew-symmetric triangular Circulant $\widetilde{SC}^{\Delta}$ is again $\mathcal{L}_{SC^{\Delta}}$.
\end{remark}

\subsection{LSD of $n^{-1/2}\widetilde{H}_n$ and $n^{-1/2}\widetilde{RC}_n$}
Simulations suggest that the LSDs of $n^{-1/2}\widetilde{H}_n$ and $n^{-1/2}\widetilde{RC}_n$ exist and are different from those of $n^{-1/2}H_n$ and  $n^{-1/2}RC_n$ respectively. See Figure \ref{fig:hankrevcirc}. We now establish this rigorously. 

In this context, \emph{symmetric} words play the key role. 
A word $w\in \mathcal{W}_{2k}$ is called symmetric if each letter in $w$ occurs
once each in an odd and an even position.  
For example, the word $aabb$ is symmetric and the word $abab$ is not.
We shall denote the set of symmetric words
of length $2k$ by $\mathcal{S}_{2k}$. 
 All Catalan
words are symmetric. An example of a non-Catalan symmetric word is $abcabc$. It is easy to prove that
\begin{equation}
\#\mathcal{S}_{2k}=k!.
\end{equation} 

\begin{figure}[!t]
\centering
\includegraphics{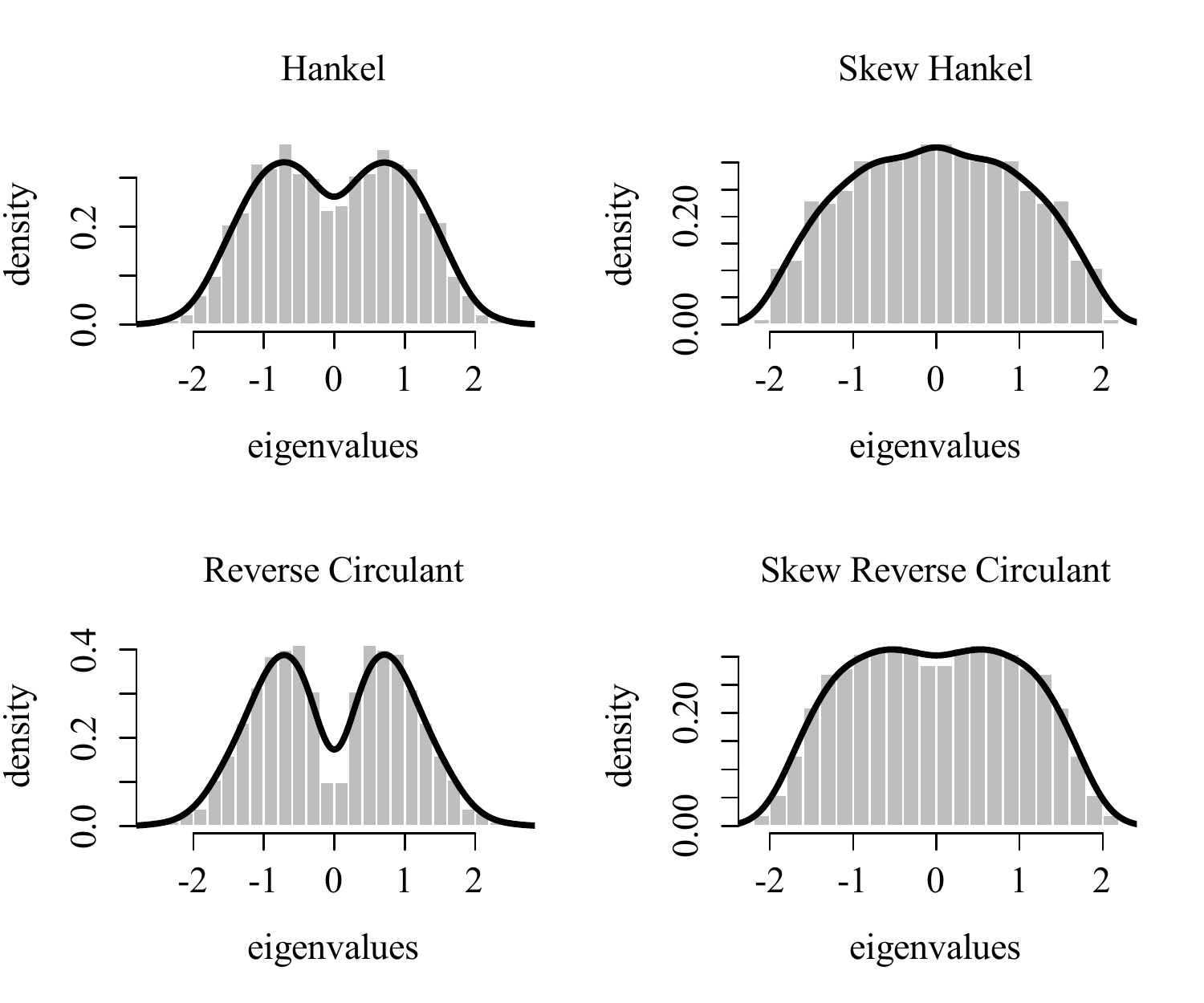}
\caption{Histograms and kernel density estimates for the ESD's of $n^{-1/2}H_n$, $n^{-1/2}\widetilde{H}_n$, $n^{-1/2}RC_n$ and $n^{-1/2}\widetilde{RC}_n$ with $n=1000$ and $\mathcal{N}(0,1)$ entries.}
\label{fig:hankrevcirc}
\end{figure}
\begin{theorem}
If the input sequence  satisfies (A1), (A2) or (A3), then the LSDs of 
$n^{-1/2}\widetilde{H}_n$ and $n^{-1/2}\widetilde{RC}_n$  exist, are universal  and are different from the LSDs of $n^{-1/2}H_n$ and  $n^{-1/2}RC_n$ respectively.
\end{theorem}
\begin{proof} We first consider the skew-symmetric Hankel. 
First suppose $w\in\mathcal{C}_{2k}$. It is known that then 
$p_H(w)=1$.
By an argument similar to that given in the proof of Theorem~\ref{thm:wig} one can show that $p_{\widetilde{H}}(w)=1$.

Now suppose $w$ is not symmetric. It is known that then 
$p_H(w)=0$.  Since, $|s_{\pi}|\leqslant 1$, it follows that for such words $p_{\widetilde{H}}(w)=0$ too.

More generally, for any pair-matched word $w$, the limit $p_{\widetilde{H}}(w)$ can be shown to exist using the same Riemann approximation technique that is used in the Hankel case (see, for example, \citet{bose2008another}). We omit the details.

We now  show that this  LSD is not same as in the symmetric Hankel case. 
 Since $|s_{\pi}|\leqslant 1$, it is clear that the limit is sub-Hankel. 
It is enough to show that $\beta_{2k}(\widetilde{H})< \beta_{2k}(H)$ for some $k\geqslant 1$. Since Catalan words contribute $1$ to both $\beta_{2k}(H)$ and $\beta_{2k}(\widetilde{H})$ and non-symmetric words do not contribute at all, we need to look a non-Catalan symmetric word. The first such word is $w=abcabc$. We shall show that $p_{\widetilde{H}}(abcabc)<\frac{1}{2}=p_{H}(abcabc)$.

 So let us consider the word $w=abcabc$ and its four generating vertices, viz., $\pi(0), \pi(1), \pi(2), \pi(3)$. Writing $\nu_i=\pi(i)/n$ and expressing the $\frac{1}{n^4}\#\Pi^*(w)$ as a Riemann sum, we know from \citet{bose2008another} that for the Hankel matrix,
\begin{equation}
p_{H}(w)=\int_{I^4} \mathbf{1}_{(0<\nu_0+\nu_1-\nu_3<1,\,  0<\nu_2+\nu_3-\nu_0<1)}d\nu_3 d\nu_2 d\nu_1 d\nu_0,
\end{equation}
where $I^4$ is the unit $4$-cube. Let $P$ be the subset of $I^4$ where the integrand above is positive. For the skew-symmetric case, however, there are many $\pi\in\Pi^*(w)$ such that $s_{\pi}=-1$, which means that there are lots of cancellations. More formally, note first that for any $\pi\in\Pi^*(w)$, we have 
\begin{eqnarray}
\nu_4=\nu_0+\nu_1-\nu_3, \\
\nu_5=\nu_2+\nu_3-\nu_0.
\end{eqnarray}
If we define
\begin{equation}
g(\nu)=s_{\pi}=(-1)^{\sum_{j=1}^{2k}\mathbf{1}_{(\nu_{j-1}<\nu_j)}},
\end{equation}
then by resorting to the Riemann approximation technique it is easy to see that
\begin{equation}
p_{\widetilde{H}}(w)=(-1)^3\int_{I^4} g(\nu)\mathbf{1}_{(0<\nu_0+\nu_1-\nu_3<1,\,  0<\nu_2+\nu_3-\nu_0<1)}d\nu_3 d\nu_2 d\nu_1 d\nu_0 .
\end{equation}
We shall show that on a subset of $P$ of positive Lebesgue measure, $g(\nu)=1$.  Consider the set $U=P\cap \{(\nu_0,\nu_1,\nu_2,\nu_3)\, \mid \, 0<\nu_0<\nu_1<\nu_2<\nu_3<1\}\subseteq I^4$. We claim that on $U$, one has $g(\nu)=1$. To see this, note that we automatically have $\nu_j-\nu_{j-1}>0$ for $j=1,2,3$. Moreover,
\begin{eqnarray}
\nu_4-\nu_3=\nu_1+\nu_0-2\nu_3<0,\\
\nu_5-\nu_4=(\nu_2-\nu_1)+2(\nu_3-\nu_0)>0,\\
\nu_6-\nu_5=2\nu_0-\nu_2-\nu_3<0.
\end{eqnarray}
Therefore, on $U$ we have, $g(\nu)=(-1)^{1+1+1+(-1)+1+(-1)}=1$. It now suffices to show that
\begin{equation}
\int_{U} \mathbf{1}_{(0<\nu_0+\nu_1-\nu_3<1,\,  0<\nu_2+\nu_3-\nu_0<1)}d\nu_3 d\nu_2 d\nu_1 d\nu_0 >0.
\end{equation}
With some easy manipulations with the constraints it is easy to show that
\begin{align*}
\int_{U} \mathbf{1}_{(0<\nu_0+\nu_1-\nu_3<1,\,  0<\nu_2+\nu_3-\nu_0<1)}d\nu & \geqslant \int_{\frac{1}{3}}^{\frac{1}{2}}\int_{\nu_0}^{\frac{1}{2}}\int_{1-\nu_1}^{\frac{1+\nu_0}{2}}\int_{\nu_2}^{1+\nu_0-\nu_2}\, d\nu_3 d\nu_2 d\nu_1 d\nu_0 \\
& =\frac{19}{62208}>0.
\end{align*}
This completes the proof for the skew-symmetric Hankel. \vskip5pt

Now consider the skew-symmetric Reverse Circulant. By following the arguments in the Hankel case, it is easy to see that each word limit exists, thereby proving the existence of the LSD. Moreover, it is known that for the Reverse Circulant, $p_{RC}(w)=1$ if $w$ is symmetric 
and is 0 otherwise. In the present case, $p_{\widetilde{RC}}(w) \leq 1$ for all symmetric words and the non-symmetric words continue to contribute zero. It is also easy to show that if $w\in \mathcal{C}_{2k}$, then $p_{\widetilde{RC}}(w)=p_{RC}(w)=1$. Thus, as before it remains to seek out a symmetric non-Catalan word $w$ such that $p(w) < 1$. Once again we may look at $w=abcabc$ and prove this. Due to the similarity with the Hankel case, we skip the details. 
\end{proof}

\section{A related class of symmetric matrices}
We have seen that skew-symmetry does not change the LSD of the Wigner, Toeplitz and the Symmetric Circulant, whereas it changes the LSD of the Hankel and the Reverse Circulant. 

Let $M_n$ be the  $n \times n$ symmetric matrix whose upper and lower triangle entries are respectively $+1$ and $-1$, the anti-diagonal consisting of $0$'s. 
Then $M_n=((m_{ij}))$ where 
\begin{equation}
m_{ij}=\begin{cases}
1, & \text{ if $i+j<n+1$,}\\
0, & \text{ if $i+j=n+1$, and}\\
-1, & \text{ if $i+j>n+1$.}
\end{cases}
\end{equation}
We show that LSD exists for the Schur-Hadamard product of $M_n$ with any of the above five matrices. For a patterned matrix $A_n$, we denote by $\widehat{A}_n$ its modified version $M_n\odot A_n$.

Note that for the Wigner and the Hankel cases, the Schur-Hadamard product is also of the same type (with a modified input sequence where the signs have changed for some elements of the sequence)--the fact that the anti-diagonal is zero does not affect the LSDs. Hence their LSDs remain unchanged due to the universality of the LSD with respect to the input variables as long as they satisfy Assumption (A1), (A2) or (A3). As we shall see, the LSD remains unchanged for the modified Reverse Circulant matrix too.

Note that $n^{-1/2}\widehat{T}_n$ and $n^{-1/2}\widehat{SC}_n$ are not Toeplitz and Symmetric Circulant matrices. We show that in each case, the LSD exists and are different from $\mathcal{L}_T$ and $\mathcal{N}(0, 1)$ respectively.
See Figure \ref{fig:rel_sym} for simulation results.
\begin{figure}[!t]
\centering
\includegraphics{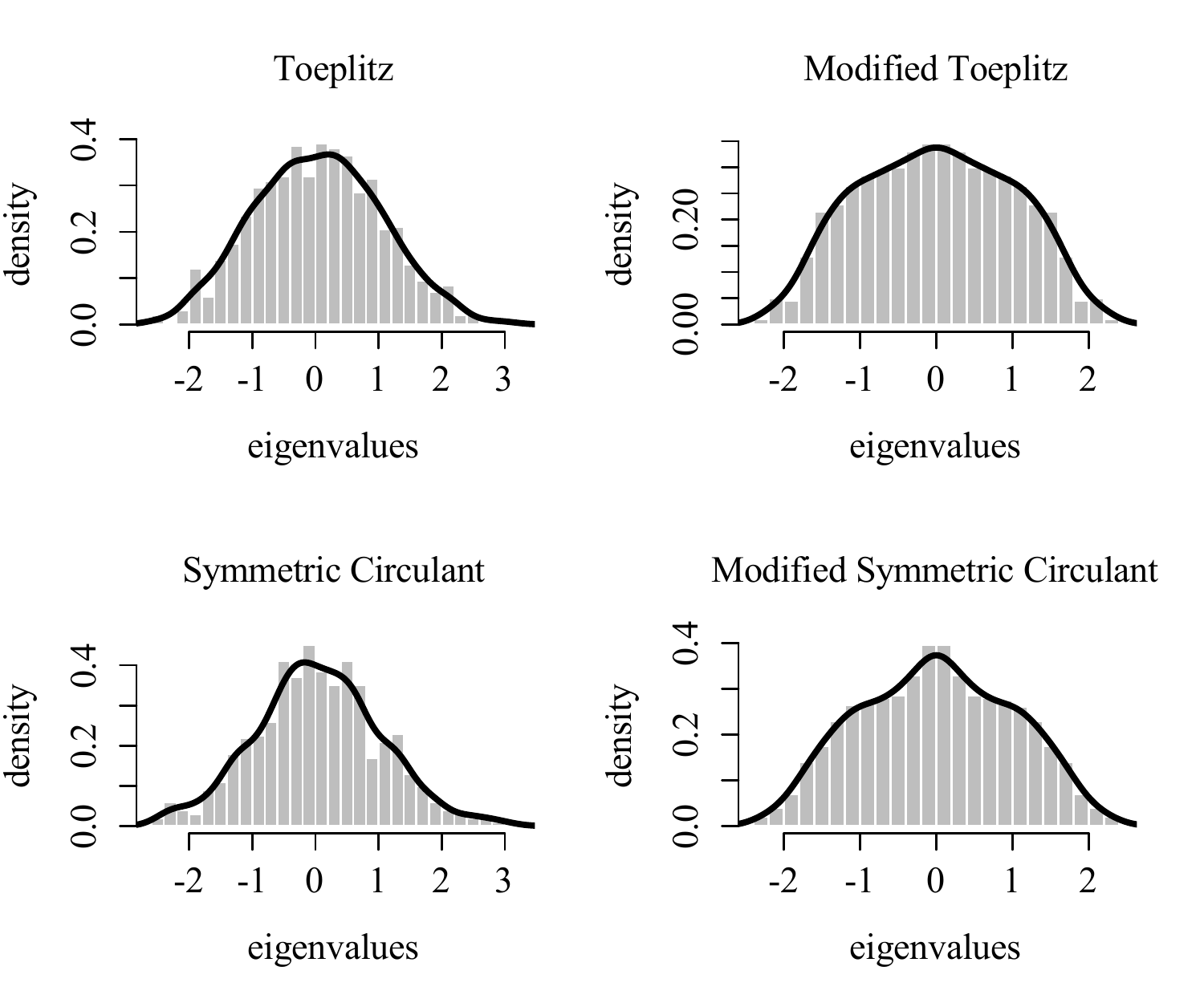}
\caption{Histograms and kernel density estmates for the ESD's of $n^{-1/2}T_n$, $n^{-1/2}\widehat{T}_n$, $n^{-1/2}SC_n$ and $n^{-1/2}\widehat{SC}_n$ with $n=1000$ and $\mathcal{N}(0,1)$ entries.}
\label{fig:rel_sym}
\end{figure}

Similar to the skew-symmetric case, define  
\begin{equation}
\epsilon_{i}=(1-\mathbf{1}_{(\pi(i-1)+\pi(i)=n+1)})(-1)^{\mathbf{1}_{(\pi(i-1)+\pi(i)>n+1)}},
\end{equation}
and
\begin{equation}
m_{\pi}=\prod_{i=1}^{n}\epsilon_i.
\end{equation}
Then we have the following analogue of (\ref{eq:emt7}):
\begin{equation}\label{eq:emt8}
\lim \E\beta_{2k}(n^{-1/2}\widehat{A}_n)=\sum_{w \in \mathcal{W}_{2k}}p_{\widehat{A}}(w),
\end{equation}
where
\[
p_{\widehat{A}}(w):=\lim_n\frac{1}{n^{1+k}}\sum_{\pi \in \Pi(w)} m_{\pi}=\lim_n\frac{1}{n^{1+k}}\sum_{\pi \in \Pi^*(w)} m_{\pi}
\]
is assumed to exist for each $w\in\mathcal{W}_{2k}$.
\subsection{LSD of $n^{-1/2}\widehat{RC}_n$}
\begin{theorem}\label{thm:modrevcirc}
If the entries of $n^{-1/2}\widehat{RC}_n$ satisfy (A1) or (A2) or (A3), then its LSD is same as the LSD of $n^{-1/2}RC_n$, i.e., $\mathcal{L}_{RC}$.
\end{theorem}
\begin{proof}
To prove this theorem, note that by (\ref{eq:emt8}) it is enough to prove that $m_{\pi}=1$ for each $\pi\in\Pi^*(w)$, for each $w\in\mathcal{W}_{2k}$.
Define 
\[
t(i)=\pi(i-1)+\pi(i)\ \ \text{and}\ \ u(i)=t(i)-(n+1).
\]
Call a circuit $\pi$ \emph{good} if  $m_{\pi}\neq0$. It is enough to consider only such circuits.

If $w[i]=w[j]$, we have 
\[
t(i)\equiv t(j) \ (\text{mod } n),
\]
which implies that $u(i) \equiv u(j) \ (\text{mod } n)$. Now note that 
\begin{equation}
-(n-1)=2-(n+1)\leqslant u(i) \leqslant n+n-(n+1)=n-1,
\end{equation}
and hence 
\begin{equation}
|u(i)-u(j)|\leqslant 2(n-1).
\end{equation}
So, we must have
\begin{equation}
u(i)-u(j)=0,\pm n.
\end{equation}
Observe that 
\begin{enumerate}
\item If $u(i)-u(j)=0$, 
then  $\epsilon_i=\epsilon_j$, which yields $\epsilon_i\epsilon_j=1$.
\item If $u(i)-u(j)=n$, then $u(i)=n+u(j)>0$, and $u(j)=u(i)-n<0$, as $|u(l)|\leqslant n-1$ for any $l$. So, in this case $\epsilon_i\epsilon_j=-1$.
\item If $u(i)-u(j)=-n$, then again $\epsilon_i\epsilon_j=-1$, by interchanging the roles of $i$ and $j$ in the previous argument.
\end{enumerate}

As a consequence
\begin{equation}
m_{\pi}=(-1)^{e_{\pi}},
\end{equation}
where $e_{\pi}$ is the number of matches $(i,j)$ in $\pi$ for which $u(i)-u(j)=t(i)-t(j)=\pm n$. Let further $e_{\pi}^+$ be  the number of matches $(i,j)$ in $\pi$ for which $t(i)-t(j)=n$ and $e_{\pi}^-=e_{\pi}-e_{\pi}^+$.
First notice that
\begin{equation}
\sum_{i=1}^{2k}t(i)=2\sum_{i=1}^{2k}\pi(i).
\end{equation}
The same sum can be written as
\begin{equation}
\sum_{(i,j)\, \text{match}}(t(i)+t(j)).
\end{equation}
Notice then that
\begin{align*} \sum_{(i,j)\, \text{match}}(t(i)+t(j)) & = \sum_{(i,j)\, \text{match}}(t(i)-t(j))+2\sum_{(i,j)\, \text{match}}t(j)\\
&=(e_{\pi}^+ - e_{\pi}^-)n + 2\sum_{(i,j)\, \text{match}}t(j)\\
&=ne_{\pi} - 2ne_{\pi}^-+ 2\sum_{(i,j)\, \text{match}}t(i).
\end{align*}
It follows from the above considerations that $ne_{\pi}$ is even always. Now suppose $n$ is odd. It then follows that $e_{\pi}$ is even and therefore $m_{\pi}=1$. The case with $n$ even seems to be more complicated. It is not clear why $e_{\pi}$ has to be even. We shall use a little trick to bypass the need to pinpoint the parity of $e_{\pi}$ in the $n$ even case. Define for $w\in \mathcal{W}_{2k}$, 
\begin{align}
& q_n(w):=\frac{1}{n^{1+k}}\sum_{\pi\in\Pi^*(w)}m_{\pi},\\
& p_n(w):=\frac{1}{n^{1+k}}\#\Pi^*(w).
\end{align}
Then it is known from \citet{bose2008another} that
\begin{equation}
p_n(w)=p_{RC}(w)+o(1),
\end{equation}
which implies, since $|q_n(w)|\leqslant |p_n(w)|$, that
\begin{equation}
q_n(w)=O(1).
\end{equation}
Now we have already proved that (as we have proved that $m_{\pi}=1$ for $n$ odd)
\begin{equation}
q_{2n+1}(w)=p_{RC}(w)+o(1).
\end{equation}
In the following lemma we shall write $\Pi^*_{n}(w)$ instead of $\Pi^*(w)$ to explicitly denote the dependence on $n$.
\begin{lemma}\label{lem1}
We have
\begin{equation}
\#\Pi^*_{n+1}(w)-\#\Pi^*_n(w)=o(n^{1+k}).
\end{equation}
\end{lemma}
\begin{proof}
We have
\begin{equation}
p_n(w)=\frac{1}{n^{1+k}}\#\Pi^*_{n}(w)=p(w)+o(1),
\end{equation}
which can be rewritten as
\begin{equation}
\#\Pi^*_n(w)=p(w)n^{1+k}+o(n^{1+k}).
\end{equation}
As a consequence
\begin{align*}
\#\Pi^*_{n+1}(w)-\#\Pi^*_n(w)&=p(w)((n+1)^{1+k}-n^{1+k})+o(n^{1+k})\\
&=p(w)O(n^k)+o(n^{1+k})\\
&=o(n^{1+k}).
\end{align*}
\end{proof}
We need another lemma.
\begin{lemma}\label{lem2}
We have
\begin{equation}
q_{n+1}(w)-q_{n}(w)=o(1).
\end{equation}
\end{lemma}
\begin{proof}
We write
\begin{align*}
&|q_{n+1}(w)-q_{n}(w)|\\
&=|\frac{1}{(n+1)^{1+k}}\sum_{\pi\in\Pi^*_{n+1}(w)}s_{\pi}-\frac{1}{n^{1+k}}\sum_{\pi\in\Pi^*_n(w)}s_{\pi}|\\
&=|\frac{1}{(n+1)^{1+k}}\sum_{\pi\in\Pi^*_{n}(w)}s_{\pi}+\frac{1}{(n+1)^{1+k}}\sum_{\pi\in\Pi^*_{n+1}(w)\setminus \Pi^*_n(w) }s_{\pi} -\frac{1}{n^{1+k}}\sum_{\pi\in\Pi^*_n(w)}s_{\pi}|\\
&\leqslant |\frac{1}{(n+1)^{1+k}}\sum_{\pi\in\Pi^*_{n}(w)}s_{\pi}-\frac{1}{n^{1+k}}\sum_{\pi\in\Pi^*_n(w)}s_{\pi}|+|\frac{1}{(n+1)^{1+k}}\sum_{\pi\in\Pi^*_{n+1}(w)\setminus \Pi^*_n(w) }s_{\pi}|\\
&\leqslant |\left(\left(\frac{n}{n+1}\right)^{1+k}-1\right)\frac{1}{n^{1+k}}\sum_{\pi\in\Pi^*_{n}(w)}s_{\pi}|+\left(\frac{n}{n+1}\right)^{1+k}\frac{1}{n^{1+k}}\#(\Pi^*_{n+1}(w)\setminus \Pi^*_n(w))\\
&=o(1)O(1)+(1+o(1))o(1) \hskip40pt \text{(by Lemma \ref{lem1})}\\
&=o(1).
\end{align*}
\end{proof}
Coming back to the original problem,  because of Lemma \ref{lem2}, now we can write
\begin{align*}
q_{2n+2}(w)&=q_{2n+1}(w)+o(1)\\
&=p_{RC}(w)+o(1).
\end{align*}
This means that
\begin{equation}
q_n(w)=p_{RC}(w)+o(1),
\end{equation}
which completes the proof of the theorem.
\end{proof}
\subsection{LSDs of $n^{-1/2}\widehat{T}_n$ and $n^{-1/2}\widehat{SC}_n$}
\begin{theorem}
If the input sequence  satisfies (A1), (A2) or (A3), then the LSDs of $n^{-1/2}\widehat{T}_n$ and $n^{-1/2}\widehat{SC}_n$  exist, are universal and are different from the LSDs of $n^{-1/2}T_n$ and  $n^{-1/2}SC_n$ respectively.
\end{theorem}
\begin{proof}
We shall outline the proof only for $n^{-1/2}\widehat{T}_n$. The proof for $n^{-1/2}\widehat{SC}_n$ is similar and is omitted. 

Once again the existence of the LSD, say $\mathcal{L}_{\widehat{T}}$,  may be proven using the Riemann approximation technique.  We show that $\mathcal{L}_{\widehat{T}}$ does not equal $\mathcal{L}_T$. As in the proof of Theorem~\ref{thm:wig} we can show that for each Catalan word $w$, $p_{\widehat{T}}(w)=1=p_{T}(w)$. Thus we need to look at a non-Catalan pair-matched word. The first such word is $w=abab$.  We shall show that  $p_{\widehat{T}}(abab)\neq p_{T}(abab)=2/3$, which would conclude proof. Using the Riemann approximation argument it is easy to show that
\[
p_{\widehat{T}}(w)=\int_{I^3} (-1)^{\sum_{i=1}^4\mathbf{1}_{(\nu_i+\nu_{i-1}>1)}}\mathbf{1}_{(0\leqslant \nu_0-\nu_1+\nu_2 \leqslant 1)} d\nu_2 d\nu_1 d\nu_0,
\]
where $\nu_3=\nu_0-\nu_1+\nu_2$ and $\nu_4=\nu_0$. Now similar to the skew-symmetric Hankel case one can show that on a subset of positive Lebesgue measure the integrand above is negative. In fact, a calculation in \texttt{Mathematica} reveals that $p_{\widehat{T}}(abab)=2/9$. This proves the theorem completely.
\end{proof}  
\bibliographystyle{apalike}
\bibliography{mybib}
\end{document}